\documentclass[11pt,reqno]{amsart}

\usepackage{eucal,mathrsfs}

\usepackage[utf8]{inputenc}

\numberwithin{equation}{section}

\usepackage[a4paper,margin=1.2in]{geometry}
                		
\usepackage{graphicx}				
										
\usepackage{amssymb}

\usepackage[utf8]{inputenc}

\usepackage[polish, english]{babel}

\usepackage[T1]{fontenc}

\usepackage{mathtools}

\usepackage{tikz-cd}

\usepackage{amsthm}

\usepackage[toc]{appendix}

\usepackage{calligra}

  \definecolor{myblue}{rgb}{0.1098, 0.1882, 0.6471}
  \definecolor{myred}{rgb}{0.6471, 0.1098, 0.1882}
  \usepackage[colorlinks,citecolor=myred,linkcolor=myred,urlcolor=myblue]{hyperref}

\usepackage{yhmath} 

\newtheorem{thm}{Theorem}[section]
\newtheorem{lem}[thm]{Lemma}

\newtheorem{prop}[thm]{Proposition}

\theoremstyle{remark}

\newtheorem{rmk}[thm]{Remark}
\newtheorem{ex}[thm]{Example}

\newcommand{\cD}{{\mathcal D}}
\newcommand{\cE}{{\mathcal E}}
\newcommand{\cF}{{\mathcal F}}

\newcommand{\cI}{{\mathcal I}}

\newcommand{\cO}{{\mathcal O}}

\newcommand{\cT}{{\mathcal T}}

\renewcommand{\AA}{{\mathbb A}}
\newcommand{\BB}{{\mathbb B}}
\newcommand{\CC}{{\mathbb C}}
\newcommand{\DD}{{\mathbb D}}  

\newcommand{\PP}{{\mathbb P}}

\newcommand{\QQ}{{\mathbb Q}}

\newcommand{\WW}{{\mathbb W}}
\newcommand{\ZZ}{{\mathbb Z}}

\newcommand{\oo}{{{\mathfrak o}_{K}}}

\newcommand{\Spec}{{\textnormal{Spec }}}
\newcommand{\coker}{{\textnormal{coker }}}

\newcommand{\mM}{{\mathscr{M}}}

\newcommand{{\D}}{{\mathscr{D}_{X}}} 

\newcommand{\DR}{\mathbf{DR}^{\bullet}} 
\newcommand{\DRX}{\mathbf{DR}^{\bullet}_{X}}

\newcommand{\Dd}{{\widehat{\mathcal{D}}}_{n}}

\newcommand{\Hom}{{\textnormal{Hom}}}
\newcommand{\RHom}{{\mathbf{R}\textnormal{Hom}}}

\title{Holonomic $\mathscr{D}$-modules on rigid analytic varieties} 
\author{Feliks Rączka}
\date{\today} 

  \address{Institute of Mathematics, Polish Academy of Sciences, ul.\ Śniadeckich 8,
    \newline\indent 00-656 Warsaw, Poland
  }

\email{fraczka@impan.pl}

\begin{document}

\begin{abstract}
We study the category of holonomic $\mathscr{D}_{X}$-modules for a quasi-compact, quasi-separated, smooth rigid analytic variety $X$ over the field $\mathbb{C}(\!(t)\!)$. In particular, we prove finiteness of the de Rham cohomology for such modules.
\end{abstract}

\maketitle

\section{Introduction}

Finiteness of variants of the de Rham cohomology is a problem with a long history. If $(X,\cO_{X})$ is a `nice' locally ringed space (for example a smooth algebraic variety over $\CC$, a complex manifold, a smooth rigid analytic variety) then one can introduce the cotangent sheaf $\Omega_{X}$ and the sheaf of differential operators $\mathscr{D}_{X}.$ Given a coherent left $\mathscr{D}_{X}$-module $\mathscr{M}$ can consider its \textit{de Rham complex}
\begin{equation}\label{de Rham complex}
\DRX(\mathscr{M})=\left[\mathscr{M}\otimes_{\cO_{X}}\Omega_{X}^{0}\to\dots\to\mathscr{M}\otimes_{\cO_{X}}\Omega^{\dim X}_{X}\right]
\end{equation}
where $\Omega_{X}^{k}=\bigwedge^{k}\Omega_{X}$. The \textit{de Rham cohomology} of $\mathscr{M}$ is defined as the (hyper)cohomology of (\ref{de Rham complex}) and is denoted as $H^{i}_{\mathrm{dR}}(X,\mathscr{M})$. If $X$ is considered over some ground field $K$ then de Rham cohomology groups are $K$-vector spaces. One is often interested in \textit{holonomic} $\mathscr{D}_{X}$-modules. This special class of $\mathscr{D}_{X}$-modules contains many classically studied objects such as vector bundles with integrable connections.

 The oldest case of finiteness of the de Rham cohomology concerns algebraic $\mathscr{D}$-modules on complex algebraic varieties. If $X$ is the affine $n$-space then the theorem of J.\ Bernstein asserts that algebraic holonomic $\mathscr{D}_{X}$-modules have finite dimensional de Rham cohomology groups. The original proof of this theorem is hard to find although some of Bernstein's ideas are presented in \cite{Bernstein}. A nice proof may be found in the book of J.\-E.\ Bj\"{o}rk \cite[Chapter 1, Theorem 6.1]{Bjork}. Bernstein's result has been later generalized to the case when $X$ is a smooth complex algebraic variety and finally to the derived setting. The strongest version says that if $f:X\to Y$ is a morphism of smooth varieties then the $\mathscr{D}$-module theoretic direct image $\int_{f}:D^{b}_{qc}(\mathscr{D}_{X})\to D^{b}_{qc}(\mathscr{D}_{Y})$ restricts to the functor
$\int_{f}:D^{b}_{h}(\mathscr{D}_{X})\to D^{b}_{h}(\mathscr{D}_{Y})$. The notation here is taken from the book of R.\ Hotta, K.\ Takeuchi, and T.\ Tanisaki ( see \cite[Theorem 3.2.3]{HTT}).

Another classical (although much less known) case is when $\mathscr{D}$ is the ring of differential operators over the ring of formal power-series over a field of characteristic zero, i.e., when $\mathscr{D}=K[\![x_{1},\dots,x_{n}]\!][\frac{\partial}{\partial x_{1}},\dots,\frac{\partial}{\partial x_{n}}]$. This theorem has been proven by A.\ van den Essen in 1980s in a series of papers (\cite{VDE1}, \cite{VDE3}, \cite{VDE4}, \cite{VDE5},  \cite{VDE2}). A very nice and short exposition of these results has been recently written by N.\ Switala \cite{Switala}.

The situation gets much more complicated when $X$ is taken to be a smooth rigid analytic variety. The main problem is that the definition of the de Rham cohomology is not well suited for the $p$-adic setting. For example if $X=\textnormal{Spa }\QQ_{p}\langle x \rangle$ then already the first de Rham cohomology group of the holonomic module corresponding to the vector bundle with integrable connection $(\cO_{X},d)$ has infinite dimension. Morally this follows from the existence of elements $f\in\QQ_{p}\langle x\rangle$ that cannot be integrated. For example one can take power series
\begin{equation}\label{nonintegrable}
f=\sum_{n\geq0}a_{n}p^{n}x^{p^{n}-1}
\end{equation}
with $|a_{n}|=1$. This problem is usually overcome by considering overconvergent coefficients. Even then the problem is very complicated. It is still an active area of research and there is no general answer similar to the theorems described above. We refer the interested reader to the work of K.\ Kedlaya \cite{Kedlaya} and E.\ Gro{\ss}e-Kl\"{o}nne \cite{Grosse1}, \cite{Grosse2} for some (highly nontrivial) special cases and the connections between the de Rham and the rigid cohomology. We also refer to the work of V.\ Ertl and A.\ Shiho \cite{Ertl} for some `non-examples'.\newline

In this paper we focus on rigid analytic varieties over a discretely valued nonarchimedean field $K$ of equal characteristic zero $K$ (e.g. $K=\CC(\!(t)\!)$). The nonarchimedean geometry in this setting has recently gained some popularity and its intersection with the theory of differential operators seems especially fruitful. The idea of considering this setting appears implicitly already in the work of P.\ Deligne (see \cite[II. \S 1]{Deligne2}) and has been used by K.\ Kedlaya to solve the famous conjecture of C.\ Sabbah (see \cite{Kedlaya2}). While the case considered by Kedlaya and others is still slightly different from our situation  (we consider derivations that are $K$-linear but it is also common to fix an isomorphism $K\simeq\CC(\!(t)\!)$ and use the action of $\frac{d}{dt}$ on the base field to obtain some extra structure) we believe that our result may find some interesting applications.

Let us illustrate the difference between the case of equal characteristic zero and the $p$-adic case. Let $X=\textnormal{Spa }K\langle x\rangle$ and let $\mathscr{M}$ be the holonomic $\mathscr{D}_{X}$-module corresponding to $(\cO_{X},d)$. Then the de Rham cohomology of $\mathscr{M}$ is just the cohomology of the complex
\[
K\langle x\rangle \xrightarrow{f\mapsto \frac{df}{dx}dx}K\langle x \rangle dx
\]
A power series $f=\sum_{n\geq0}a_{n}x^{n}\in K[\![x]\!]$ is an element of $K\langle x\rangle$ if and only if $\lim_{n\to\infty}|a_{n}|=0$. If the residue characteristic of $K$ is zero then $|n|=1$ for any $n\in\ZZ\setminus\{0\}$ (so in particular power series of form (\ref{nonintegrable}) do not belong to $K\langle x\rangle$)). Therefore $\sum_{n\geq0}\frac{a_{n}}{n+1}x^{n+1}\in K\langle x \rangle$ if $f\in K\langle x \rangle$ and we have $H^{1}_{\mathrm{dR}}(X,\mathscr{M})=0$. This suggests that at the usual de Rham cohomology might be well-suited for the rigid analytic varieties over nonarchimedean fields of equal characteristic zero. We show that this is indeed the case.
\begin{thm}\label{MainThm}
Let $X$ be a quasi-compact, quasi-separated, smooth rigid analytic space over a discretely valued nonarchimedean field of equal characteristic zero. Then for any holonomic $\mathscr{D}_{X}$-module $\mathscr{M}$ and for all $i$ we have $\dim_{K}H^{i}_{\mathrm{dR}}(X,\mathscr{M})<\infty.$
\end{thm}
Let us briefly explain the main ideas of proof of the above theorem. The logical structure is explained in Figure \ref{ProofStructure} below.
\begin{figure}[h]\label{ProofStructure}
    \centering
    \includegraphics[scale=0.9]{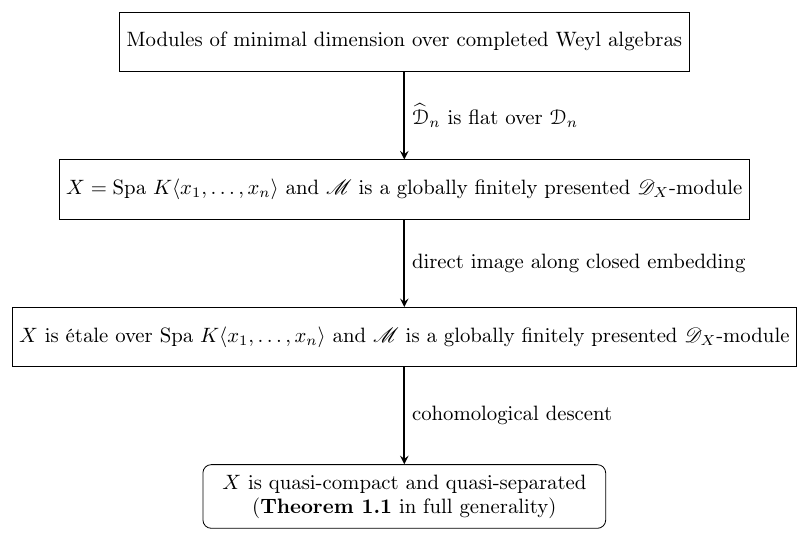}
    \caption{The logical structure of the proof of Theorem \ref{MainThm}.}
    \label{fig:enter-label}
\end{figure}
We divide the proof into several cases, each case being more general than the previous one. The point of the departure is our earlier work \cite{Felek} on modules of minimal dimension over completed Weyl algebras. If $X=\textnormal{Spa }K\langle x_{1},\dots, x_{n}\rangle$ is the Tate polydisc then we set $\mathcal{D}_{n}=\Gamma(X,\mathscr{D}_{X})$
and we write $\Dd$ for the completion of this ring with respect to the operator norm. It turns out that the natural map $\cD_{n}\to\Dd$ is flat. This observation allows us to deduce Theorem \ref{MainThm} for globally generated $\mathscr{D}_{X}$-modules on Tate's polydiscs from the main result of \cite{Felek} (see Proposition~\ref{TateDisc}).

Next we prove Theorem \ref{MainThm} for a globally generated holonomic $\mathscr{D}_{X}$-momdule on a smooth affinoid variety $X$ with a global coordinate system (i.e., an {\'e}tale morphism to some polydisc). Since $X$ is affinoid there exists a Zariski closed embedding $i:X\hookrightarrow Y$ of $X$ into a polydisc $Y$. Then we study the $\mathscr{D}$-module theoretic direct image $i_{+}\mathscr{M}$ to conclude Theorem \ref{MainThm} reducing to the previous case (see Proposition \ref{GlobalCoordinates}).

Finally, to prove Theorem \ref{MainThm} in full generality we cover $X$ by open affinoid subsets such that on each of these subsets assumptions of the previous case hold. We then use appropriate spectral sequence to conclude finiteness of the de Rham cohomology from its finiteness on each of the open subsets. 

The reduction of the problem described above justifies the generality in which Theorem~\ref{MainThm} is stated, i.e., the use of the holonomic $\mathscr{D}$-modules. Although in practice one is usually interested in the de Rham cohomology of vector bundles with integrable connections or even just the de Rham cohomology of the trivial vector bundle $(\cO_{X},d)$ the approach sketched above quickly leads to the more general category of holonomic $\mathscr{D}$-modules. Indeed, for a closed embedding $i:X\hookrightarrow Y$ the direct image $i_{+}(\cO_{X},d)$ will be a holonomic $\mathscr{D}_{Y}$-module that is not a vector bundle as it is supported on a proper closed subset. On the other hand, it is hard to think of a larger category of $\mathscr{D}$-modules that is natural to work with and enjoys the property of the finiteness of the de Rham cohomology, although it is worth mentioning that as a byproduct of our proof one can fairly easily construct examples of non-holonomic $\mathscr{D}$-modules with finite de Rham cohomology. One such example will be discussed in Example \ref{Nonholonomic}. \newline

Since the idea sketched in the Figure \ref{ProofStructure} is rather standard and similar to known proofs of finiteness of the de Rham cohomology in the algebraic case it is worthwhile to explain the main difficulties and differences from the known cases. This can be seen in the following example.

Let $K=\CC(\!(t)\!)$ $X=\textnormal{Spa }K\langle x\rangle$ be the Tate disc. Assume that we are given a vector bundle with integrable connection on $X$, i.e., a free $K\langle x\rangle$-module $M$ of finite rank together with a $K$-linear map $\nabla:M\to M$ that is continuous for the canonical topology on $M$ and satisfies \textit{Leibniz's rule} 
\[
\nabla(f.m)=\frac{df}{dx}.m+f.\nabla{m}.
\]
Then we have
\[
H^{0}_{\mathrm{dR}}(X,M)=\ker \nabla,\quad H^{1}_{\mathrm{dR}}(X,M)=\coker \nabla.
\]
A natural approach to the problem of finiteness of the de Rham cohomology of $(M,\nabla)$ would be the following. First we look for a \textit{model}, i.e., a finitely generated $\oo\langle x\rangle$ submodule $M^{\circ}\subset M$ such that $M^{\circ}\otimes_{\oo}K=M$ and $\nabla(M^{\circ})\subset M^{\circ}$. Then we consider the \textit{reduction} of our model, i.e. the $\CC[x]$-module $\overline{M}=M\otimes_{\oo}\CC$ together with the induced connection $\overline{\nabla}$. Now $(\overline{M},\overline{\nabla})$ is a vector bundle with connection on the affine line over $\CC$ and therefore it has finite dimensional de Rham cohomology by the classical theory. One can show (see \cite[Lemma 3.3]{Felek}) that finiteness of the de Rham cohomology of $(M,\nabla)$ follows from finiteness of the de Rham cohomology of $(\overline{M},\overline{\nabla})$ and moreover if we write $\chi_{\mathrm{dR}}$ for the Euler characteristic with respect to the de Rham cohomology then
\[
\chi_{\mathrm{dR}}(X,(M,\nabla))=\chi_{\mathrm{dR}}(\AA^{1}_{\CC},(\overline{M},\overline{\nabla})).
\]
Following this idea we now ask if given $(M,\nabla)$ one can always find a model. The module $M$ carries a canonical family of equivalent norms and it is easy to see that if $(M,\nabla)$ has a model then the spectral radius of $\nabla$ satisfies $|\nabla|_{\textnormal{sp},M}\leq 1$ (see \cite[Definition 6.1.3]{Kedlaya3} for the notation). If we take, e.g.,
\begin{equation}\label{Przyklad}
M=K\langle x\rangle e, \quad \nabla(f.e)=\left(\frac{df}{dx}-t^{-1}f\right).e
\end{equation}
then $|\nabla|_{\textnormal{sp},M}=|t^{-1}|>1$ and $(M,\nabla)$ does not admit a model. This is where the completed Weyl algebra shows up. We consider two rings
\[
\cD=K\langle x\rangle [\partial],\quad\textnormal{and} \quad \widehat{\cD}=K\langle x,\partial\rangle.
\]
The former is the ring of differential operators on $X$ and the latter (the completed Weyl algebra) is its completion with respect to the operator norm. The elements of $\cD$ are represented as polynomials $\sum f_{i}\partial^{i}$ with $f_{i}\in K\langle x\rangle $ and the elements of $\widehat{\cD}$ are represented as formal power series  $\sum f_{i}\partial^{i}$ such that $\lim |f_{i}|=0$. Any $(M,\nabla)$ can be seen as a left $\cD$-module with $\partial.m=\nabla(m)$. If $|\nabla|_{\textnormal{sp},M}\leq 1$ then it is in fact a $\widehat{\cD}$-module. This suggest a way of forcing $M$ to have a model by setting
\[
\widehat{M}=\widehat{\cD}\otimes_{\cD}M,
\]
and studying $\widehat{M}$ instead. It is a small miracle (to which Section \ref{Polydiscs} is devoted) that this base-change operation preserves the dimensions of the de Rham cohomology groups. For example, if we take $(M,\nabla)$ defined in (\ref{Przyklad}) then the corresponding $\cD$-module is
\[
M=\cD/\cD(\partial-t^{-1})=\cD/\cD(1-t\partial)
\]
Note that $1-t\partial$ is a unit in $\widehat{\cD}$ and therefore $\widehat{M}=0$. This is fine since $\ker\nabla=0$ as $|\nabla(f)|=|t^{-f}||f|$ and $\coker\nabla=0$ as 
\begin{equation*}
\nabla\left(\sum_{i\geq0}t^{i}\partial^{i}(f)\right)=\sum_{i\geq0}t^{i}\partial^{i+1}(f)-t^{-1}\sum_{i\geq0}t^{i}\partial^{i}(f)=t^{-1}f
\end{equation*}
The price that we pay for passing from $M$ to $\widehat{M}$ is that $\widehat{M}$ will not be a vector bundle in general. In fact it is rarely finitely generated over $\cD$. On the other hand, if $M$ is a $\cD$-module of minimal dimension (i.e., an algebraic object corresponding to a holonomic $\mathscr{D}_{X}$-module) then $\widehat{M}$ is a $\widehat{\cD}$-module of minimal dimension (see Lemma \ref{CompletedWeyl}). The discussion above in the case of Tate's disc carries to arbitrary dimension and therefore to realize the first step of our proof of Theorem \ref{MainThm} we need to study modules of minimal dimension over completed Weyl algebras. This is precisely the goal of our previous paper \cite{Felek} which should be seen as the first part of this article.\newline

It would be interesting to compare our result with other known variants of the de Rham cohomology used in the nonarchimedean setting. We expect an analogue of our Theorem \ref{MainThm} to be also valid (with $K$ of equal characteristic zero) for the overconvergent de Rham cohomology used by the $p$-adic geometers and for the de Rham cohomology of the $\wideparen{D}$-modules from the recent work of K.\ Ardakov, A.\ Bode, and S.\ Wadsley (\cite{Ardakov1}, \cite{Ardakov2}, \cite{Ardakov3}). Perhaps such results could be derived directly from our Theorem \ref{MainThm}. It is also natural to ask if the assumption that $K$ is discretely valued can be dropped. We believe that it is true but it requires revisiting the proofs in \cite{Felek} and it seems to require some work. Finally we mention the technical problem of $D$-affinity which will be discussed in greater generality in Subsection \ref{Daffinity}. We think that a positive answer to this problem would make our proof less technical and easier to grasp.

\subsection*{Acknowledgements}
This work was supported by the project KAPIBARA funded by the European Research Council (ERC) under the European Union's Horizon 2020 research and innovation programme (grant agreement No 802787). We thank P.\ Achinger and A.\ Langer for multiple discussions and for their interest in this work. We also thank V.\ Ertl for explaining to us the recent state of art on the de Rham cohomology in the $p$-adic case and for providing suitable references.

\section{Preliminary results on $\mathscr{D}$-modules on rigid analytic spaces}\label{Preliminaries}

In this section we fix the notation and recall basic definitions concerning rigid analytic varieties and $\mathscr{D}$-modules on such varieties. In the part regarding general nonarchimedean geometry we mostly rely on the standard literature, i.e, the books of Huber \cite{Huber}, Bosch--G{\"u}ntzer--Remmert \cite{BGR}, and Fresnel--Van der Put \cite{FVdP}.  When it comes to $\mathscr{D}$-modules the situation is slightly more complicated. For the basic definitions we rely on the work of Mebkhout \cite{Meb1} and his joint paper with Narv{\'a}ez Macarro \cite{Meb2}. Unfortunately, we are not aware of a book which deals with the $\mathscr{D}$-modules on rigid analytic varieties in greater detail and covers the material needed in this paper. We therefore refer to the book of Hotta--Takeuchi--Tanisaki \cite{HTT}, which covers the theory of algebraic $\mathscr{D}$-modules on complex analytic varieties. We believe that the properties discussed in this section are so elementary that the interested reader will have no difficulties in translating the definitions and the proofs given in \cite{HTT} to the nonarchimedean setting.

\subsection{Rigid analytic varieties}

Since there are multiple formalisms commonly used to describe rigid analytic varieties we start by specifying the mathematical language used in this paper.

From now on we fix a discretely valued nonarchimedean field $K$ of equal characteristic zero. We write $\oo$ for the ring of integers and $k$ for the residue field. If $\varpi\in\oo$ is a uniformizer then $K\simeq k(\!(\varpi)\!)$. Since finiteness of the de Rham cohomology may be checked after a base change along a field extension we can always assume that $k$ is algebraically closed. Therefore the reader will not lose much generality assuming that $K=\CC(\!(t)\!)$ is the field of formal Laurent series over $\CC$.

A \textit{rigid analytic variety} over $K$ is an adic space in the sense of Huber \cite[Chapter 1]{Huber} that is of finite type over $\textnormal{Spa}(K,\oo)$. As such, it is locally of the form ${\rm Spa}(A, A^\circ)$ where $A$ is an affinoid $K$-algebra (i.e., $A$ is isomorphic to $K\langle x_{1},\dots,x_{n}\rangle/I$) and $A^\circ\subset A$ is the ring of powerbounded elements. In what follows we write $\textnormal{Spa }A$ to denote ${\rm Spa}(A, A^\circ)$. We say that $X$ is \textit{smooth} if it is smooth in the sense of \cite[Chapter 1.6]{Huber}. By
\[
\BB^{n}=\textnormal{Spa }K\langle x_{1},\dots,x_{n}\rangle
\]
we denote the $n$-dimensional \textit{Tate polydisc}.

If $X$ is a smooth rigid analytic variety then we have a well defined cotangent sheaf $\Omega_{X}$ (cf. \cite[Chapter 3.6]{FVdP}), which is locally free of rank $n=\dim X$. The tangent sheaf is defined as its $\cO_{X}$-dual and denoted $\cT_{X}$. If $U\subset X$ is an open affinoid subset then we say that the elements $x_{1},\dots,x_{n}\in\cO_{X}(U)$ form a \textit{local coordinate system} on $U$ if $\Omega_{X|U}$ is a free $\cO_{U}$-module and the elements $dx_{1},\dots,dx_{n}\in\Omega_{X}(U)$ form a basis for $\Omega_{X|U}$. In this case we also say that $U$ \textit{admits a coordinate system}. If $X$ itself admits a coordinate system we say that it \textit{admits a global coordinate system}. Note that by the definition this implies that $X$ is affinoid. A coordinate system on $X$ is nothing else than an {\'e}tale morphism $X\to \BB^{n}$. Since every point $x\in X$ has a neighbourhood that admits such morphism we see that open subsets of $X$ admitting a coordinate system form a basis for the topology on $X$. These definitions may be generalized to the relative case when $i:X\hookrightarrow Y$ is a Zariski closed embedding. We say that such embedding admits a coordinate system if there exists a coordinate system $y_{1},\dots,y_{n}$ on $Y$ such that $X$ is cut out by the ideal $\cI=(y_{r+1},\dots,y_{n})$ and the images of $y_{1},\dots,y_{r}$ in $\cO_{Y}/\cI=\cO_{X}$ form a coordinate system on $X$. The following lemma has been taken from the notes of B.~\ Zavyalov \cite[Lemma 5.8]{Zavyalov} (see also \cite[\href{https://stacks.math.columbia.edu/tag/0FUE}{Tag 0FUE}]{stacks-project} for an analogous statement for algebraic schemes).

\begin{lem}\label{Zavyalov}
Let $i:X\hookrightarrow Y$ be a Zariski closed immersion of smooth rigid analytic varieties. Then for every $x\in X$ there exists an affinoid open $x\in U_{x}\subset Y$ and an {\'e}tale morphism $h:U_{x}\to \BB^{\dim Y}$ such that the following diagram is cartesian
\[
\begin{tikzcd}
U_{x}\cap X\arrow{r}\arrow{d}&\BB^{\dim X}\arrow{d}\\
U_{x}\arrow{r}&\BB^{\dim Y}
\end{tikzcd}
\]
The vertical arrow on the left is induced by $i$ and the vertical arrow on the right is the inclusion of the vanishing locus of the first $(\dim Y-\dim X)$ coordinates.
\end{lem}

Lemma \ref{Zavyalov} implies that for every Zariski closed embedding $X\hookrightarrow Y$ there exists an open covering  $\{U_{i}\}$ of $Y$ such that embeddings $X\cap U_{i}\hookrightarrow U_{i}$ admit  coordinate systems for all $i$. We use this basic observation several times throughout the text without mentioning it explicitly.

\begin{rmk}
Much of the foundational work in the theory of $\mathscr{D}$-modules on rigid analytic varieties has been done by Z.\ Mebkhout and L.\ Narv{\'a}ez Macarro in their paper \cite{Meb2}. This work predates the work of Huber and has been written using the formalism of Tate, where open coverings are replaced by the so called admissible coverings. By \cite[1.1.11]{Huber} any affinoid variety in the sense of Tate can be seen as an affinoid variety in the sense of Huber and any admissible covering corresponds to an open covering. Therefore the work of Mebkhout and Narv{\'a}ez Macarro translates without changes to the formalism of Huber.
\end{rmk}

\subsection{$\mathscr{D}$-modules on rigid analytic varieties}

Let $A$ be a $K$-algebra. Then the ring of differential operators is defined as the subring
\[
\mathcal{D}_{A}=\bigoplus_{n\geq0}\mathcal{D}_{A}^{\leq n}\subset\textnormal{Hom}_{K}(A,A)
\]
where $\cD_{A}^{\leq0}=A$ and
\begin{equation}\label{DiffFiltration1}
\cD_{A}^{\leq n}=\{P\in \textnormal{End}_{K}(A):[P,f]\in\cD_{A}^{\leq n-1}\textnormal{ for all }f\in A\}
\end{equation}
We recall the following theorem of Matsumura. 
\begin{lem}[{Matsumura, \cite[1.2.2]{Meb2}}]\label{Matsumura}
Let $A$ be a noetherian $K$-algebra such that

\begin{enumerate}

\item $A$ is of equal dimension $n$.

\item Residue fields for maximal ideals are algebraic extensions of $K$.

\item There exist $x_{1},\dots,x_{n}\in A$ and $\partial_{1},\dots,\partial_{n}\in \textnormal{Der}_{K}(A,A)$ with $\partial_{i}(x_{j})=\delta_{ij}$.

\end{enumerate}
Then $\textnormal{Der}_{K}(A,A)$ is a free $A$ module with a basis $\partial_{1},\dots,\partial_{n}$.

\end{lem}

Mebkhout and Narv{\'a}ez Macarro \cite{Meb2} conclude from Lemma \ref{Matsumura} above that we have a direct sum decomposition
\begin{equation}\label{dsd}
\cD_{A}^{\leq n}=\bigoplus_{|\alpha|\leq n}A.\partial^{\alpha}.
\end{equation}
We often abuse the notation and write the equality (\ref{dsd}) as $\cD_{A}=A[\partial_{1},\dots,\partial_{n}]$. More generally, if $B$ is any $K$-algebra and $\delta_{1},\dots,\delta_{k}$ are some $K$-linear derivations of $B$ then we write $B[\delta_{1},\dots,\delta_{k}]$ for the subring of $\textnormal{End}_{K}(B)$ generated by $B$ (acting on itself by left multiplication) and these derivations.

It follows from Lemma \ref{Matsumura} that every $K$ differential operator on $A$ is continuous with respect to its natural topology. Note also that under assumptions of the Lemma above, the ring $\cD_{A}$ has a (non-canonical) involution
\begin{equation}\label{involution}
t:\cD_{A}\to\cD_{A}^{op}:\sum_{\alpha} f_{\alpha}\partial^{\alpha}\mapsto\sum_{\alpha}(-1)^{|\alpha|}\partial^{\alpha}f_{\alpha}
\end{equation}
We use without mentioning the well known fact (see \cite[Th{\'e}or{\`e}me 1.4.4]{Meb2}) that under the assumptions of Lemma \ref{Matsumura} the algebra $\cD_{A}$ is both left and right noetherian and it has a finite global homological dimension $n$.

If $M$ is a left $\mathcal{D}_{A}$-module then its \textit{de Rham complex} is defined as
\begin{equation}\label{de Rham complex 2}
\DR_{\cD_{A}}(M)=\left[M\to \bigoplus_{i=1}^{n}Mdx_{i}\to\bigoplus_{i<j}Mdx_{i}\wedge dx_{j}\to\dots\to Mdx_{1}\wedge dx_{2}\wedge\dots\wedge dx_{n}\right]
\end{equation}
with the differential given by
\[
\delta^{k}(m.dx_{I})=\sum_{j=1}^{n}\partial_{j}m.dx_{j}\wedge dx_{I}
\]
The \textit{de Rham cohomology} of $M$ is defined as the cohomology of this complex and denoted $H_{\mathrm{dR}}^{i}(M)$.\newline

If $X$ is a smooth rigid analytic variety we define sheaves of differential operators of order $\leq n$ on $X$ as a sheafified version of (\ref{DiffFiltration1}), i.e., as
\[
\mathscr{D}_{X}^{\leq n}=\{P\in \cE\textnormal{nd}_{K}(\cO_{X}): [P,f]\in\mathscr{D}_{X}^{\leq n-1}\textnormal{ for all }f\in\cO_{X} \}.
\]
These sheaves are coherent because if $x_{1},\dots, x_{n}$ is a coordinate system on an open subset $U\subset X$ and $\partial_{1},\dots,\partial_{n}$ is the dual basis to the basis $dx_{1},\dots,dx_{n}$ then 
$\mathscr{D}^{\leq n}_{U}$ is the coherent sheaf associated to the finitely generated module $\cD_{\cO_{X}(U)}^{\leq n}$ Finally, we define the sheaf of differential operators on $X$ as
\[
\mathscr{D}_{X}=\bigcup_{n\geq0}\mathscr{D}_{X}^{\leq n}.
\]
Alternatively $\mathscr{D}_{X}$ can be defined as the subsheaf of $\cE nd_{K}(\cO_{X})$ generated as a sheaf of $K$-algebras by $\cO_{X}$ and $\cT_{X}$. This follows from Lemma \ref{Matsumura}.

The notion of a coherent and globally finitely presented left (or right) $\mathscr{D}_{X}$-module is clear. A $\mathscr{D}_{X}$-module is \textit{globally finitely presented} if it admits a finite presentation
\[
\mathscr{D}_{X}^{\oplus a_{2}}\to\mathscr{D}_{X}^{\oplus a_{1}}\to\mathscr{M}\to0
\]
and it is \textit{coherent} if there exists an open covering $X=\bigcup_{i}U_{i}$ such that $\mathscr{M}_{|U_{i}}$ is finitely presented for all $i$.  If $\mathscr{M}$ is a coherent left $\mathscr{D}_{X}$-module then we define its \textit{de Rham complex} as the complex (\ref{de Rham complex}), i.e., as 
\begin{equation}\label{deRhamComplex1}
\DRX(\mathscr{M})=\left[\mathscr{M}\otimes_{\cO_{X}}\Omega_{X}^{0}\to\dots\to\mathscr{M}\otimes_{\cO_{X}}\Omega^{\dim X}_{X}\right]
\end{equation}
with the differential defined (locally, in a coordinate system) as
\begin{equation}\label{deRhamdiff}
d(m\otimes\omega)=m\otimes d\omega+\sum_{i=1}^{n}\partial_{i}m\otimes(dx_{i}\wedge\omega)
\end{equation}
A standard calculation shows that formula (\ref{deRhamdiff}) is independent of the choice of coordinates and it extends to a globally defined differential in the complex (\ref{deRhamComplex1}).
The \textit{de Rham cohomology} of $\mathscr{M}$ is defined as the (hyper)cohomology of (\ref{deRhamComplex1}) and denoted $H^{i}_{\mathrm{dR}}(X,\mathscr{M})$.

The sheaf $\cO_{X}$ is a left $\mathscr{D}_{X}$-module in a natural way. It admits a locally free resolution by the \textit{Spencer complex} (cf. \cite[Lemma 1.5.27]{HTT})
\begin{equation}\label{SpencerComplex}
\textbf{Sp}_{X}^{\bullet}(\mathscr{D}_{X})=\left[\mathscr{D}_{X}\otimes_{\cO_{X}}\bigwedge^{n}\cT_{X}\to\dots\to\mathscr{D}_{X}\otimes_{\cO_{X}}\bigwedge^{1}\cT_{X}\to\mathscr{D}_{X}\otimes_{\cO_{X}}\bigwedge^{0}\cT_{X}\right]
\end{equation}
with the differential given by
\begin{align*}
P\otimes\theta_{1}\wedge\dots\wedge\theta_{k}&\mapsto
\sum_{i}(-1)^{i+1}P\theta_{i}\otimes\theta_{1}\wedge\dots\wedge\widehat{\theta_{i}}\wedge\dots\wedge\theta_{k}\\&+\sum_{i<j}(-1)^{i+j}P\otimes[\theta_{i},\theta_{j}]\wedge\theta_{1}\wedge\dots\wedge\widehat{\theta_{i}}\wedge\dots\wedge\widehat{\theta_{j}}\wedge\dots\wedge\theta_{k}.
\end{align*}
Our convention is that the complex (\ref{SpencerComplex}) is concentrated in degrees $[-n,0]$. For any coherent left $\mathscr{D}_{X}$ module $\mathscr{M}$ we have natural isomorphisms
\begin{equation*}
\begin{split}
\mathscr{M}\otimes_{\cO_{X}}\Omega_{X}^{k}
&=\textnormal{Hom}_{\cO_{X}}(\cO_{X},\mathscr{M})\\
&=\textnormal{Hom}_{\cO_{X}}(\bigwedge^{k}\cT_{X},\mathscr{M})\\
&=\textnormal{Hom}_{\mathscr{D}_{X}}(\mathscr{D}_{X}\otimes_{\cO_{X}}\bigwedge^{k}\cT_{X},\mathscr{M})\\
&=\textnormal{Hom}_{\mathscr{D}_{X}}(\textbf{Sp}^{-k}(\mathscr{D}_{X}),\mathscr{M}).
\end{split}
\end{equation*}
It is a matter of a standard computation in local coordinates that these isomorphism yield an isomorphism of complexes
\begin{equation}\label{SpencerDR}
\DR_{X}(\mathscr{M})=\textnormal{Hom}_{\mathscr{D}_{X}}(\textbf{Sp}_{X}^{-\bullet}(\mathscr{D}_{X}),\mathscr{M}).
\end{equation}

\begin{rmk}\label{Integrability1}
Since the ring $\mathscr{D}_{X}$ is generated by $\cO_{X}$ and $\cT_{X}$, giving a left $\mathscr{D}_{X}$-module is equivalent to an $\cO_{X}$-module $\mathscr{M}$ together with a $K$-linear map $\nabla:\mathscr{M}\to\mathscr{M}\otimes_{\cO_{X}}\Omega_{X}$ satisfying the \textit{Leibniz rule}
\[
\nabla(fm)=m\otimes df+f\nabla(m)
\]
and the \textit{integrability condition} $\nabla^{2}=0$. The latter means that if we consider maps
\[
\nabla_{i}:\mathscr{M}\otimes_{\cO_{X}}\Omega^{i}_{X}\to\mathscr{M}\otimes_{\cO_{X}}\Omega^{i+1}_{X};\quad m\otimes \alpha\mapsto m\otimes d
\alpha+\nabla(m)\wedge \alpha
\]
then $\nabla_{i+1}\circ\nabla_{i}=0$. In this description $\nabla_{i}$ is the globally defined differential in the de Rham complex (\ref{deRhamComplex1}) and it is easy to verify that this description agrees with the one given by the formula (\ref{deRhamdiff}). Note that the notation used is slightly different from the usual one, where one considers the map $\nabla:\mathscr{M}\to\Omega_{X}\otimes_{\cO_{X}}\mathscr{M}$ and then the formula for $\nabla_{i}$ reads $\nabla_{i}(\alpha\otimes m)=d\alpha\otimes m+(-1)^{i}\alpha\wedge\nabla(m).$ The reason for that is purely aesthetic. After passing to local coordinates we prefer to work with elements of $\mathscr{M}\otimes_{\cO_{X}}\Omega^{i}_{X}$ written as $m.dx_{I}$ rather than with elements of $\Omega_{X}^{i}\otimes_{\cO_{X}}\mathscr{M}$ written $d_{I}.m$.
\end{rmk}

Let $X$ be a smooth affinoid variety and let $M$ be a finitely generated left $\mathscr{D}_{X}(X)$-module. We define a presheaf on affinoid subdomains
\[
\widetilde{M}:U\mapsto\mathscr{D}_{X}(U)\otimes_{\mathscr{D}_{X}(X)}M
\]

\begin{lem}\label{AffinoidCorrespondence}
Assume that $X=\textnormal{Spa A}$ admits a global coordinate system. Then
\begin{enumerate}

\item for every finitely generated left $\cD_{A}$-module $M$ the presheaf $\widetilde{M}$ is a sheaf.

\item $H^{i}(X,\widetilde{M})=0$ for $i>0$.

\item $M\mapsto\widetilde{M}$ establishes an equivalence of categories between finitely generated $\cD_{A}$-modules and globally finitely presented $\mathscr{D}_{X}$-modules. The quasi-inverse is given by $\mM\mapsto\Gamma(X,\mM)$.

\item $H^{i}_{\mathrm{dR}}(M)=H^{i}_{\mathrm{dR}}(X,\widetilde{M})$.

\end{enumerate}

\end{lem}

\begin{proof}
Clearly $A$ satisfies assumptions of Lemma \ref{Matsumura} and for every affinoid subdomain $U\subset X$ we have $\mathscr{D}_{X}(U)=\cD_{A}=\cO_{X}(U)[\partial_{1},\dots,\partial_{n}].$ We conclude parts (1)--(3) of the lemma from Tate's acyclicity theorem \cite[8.2.1]{BGR}. Let $X=\bigcup_{i=1}^{N}U_{i}$ be a finite covering of $X$ by affinoid subdomains. Then Tate's acyclicity theorem states that the \v{C}ech complex 
\[
\mathscr{C}^{\bullet}(\{U_{i}\},\cO_{X})=\left[0\to\cO_{X}(X)\to\bigoplus_{i=1}^{n}\cO_{X}(U_{i})\to\cdots\right]
\]
is exact. On the other hand under our assumptions
\[
\mathscr{D}_{X}(U)\otimes_{\mathscr{D}_{X}(X)}M=\cO_{X}(U)\otimes_{A}M.
\]
Since the maps $A\to\cO_{X}(U)$ are known to be flat (see \cite[Chapter 7.3.2, Corollary 6]{BGR}) we deduce that the complex
\[
\mathscr{C}^{\bullet}(\{U_{i}\},\widetilde{M})=\mathscr{C}^{\bullet}(\{U_{i}\},\cO_{X})\otimes_{A}M
\]
is exact. This shows that the presheaf $\widetilde{M}$ is in fact a sheaf and it has no higher sheaf cohomology, i.e., that (1) and (2) hold.

To prove (3) we show that our functor is fully faithful and essentially surjective. Let $M$ be a finitely generated $\cD_{A}$-module. By noetherianity it is finitely presented and we have a presentation
\[
\cD_{A}^{\oplus a_{2}}\to\cD_{A}^{\oplus a_{1}}\to M\to0
\]
which induces a presentation
\[
\mathscr{D}_{X}^{\oplus a_{2}}\to\mathscr{D}_{X}^{\oplus a_{1}}\to\widetilde{M}\to 0
\]
Note that the latter is really a presentation because we have
\[
\mathscr{D}_{X}(U)\otimes_{\cD_{A}}M=(\cO_{X}(U)\otimes_{A}\cD_{A})\otimes_{\cD_{A}}M=\cO_{X}(U)\otimes_{A}M
\]
and therefore $M\mapsto\widetilde{M}$ is an exact functor as the functor $\cO_{X}(U)\otimes_{A}(-)$ is exact. Now we have
\[
\textnormal{Hom}_{\mathscr{D}_{X}}(\mathscr{D}_{X},\widetilde{N})=\widetilde{N}(X)=N
\]
and therefore both $\textnormal{Hom}_{\mathscr{D}_{X}}(\widetilde{M},\widetilde{N})$ and $\textnormal{Hom}_{\cD_{A}}(M,N)$ can be realized as the kernel of the same homomorphism
\[
N^{\oplus a_{1}}\to N^{\oplus a_{2}}
\]
and therefore our functor is fully faithful. To check that it is essentially surjective we use exactness once again. Let
\[
\mathscr{D}_{X}^{\oplus a_{2}}\to\mathscr{D}_{X}^{\oplus a_{1}}\to\mathscr{M}\to 0
\]
and let $M$ be the cokernel of the corresponding map $\mathscr{D}_{X}(X)^{\oplus a_{2}}\to\mathscr{D}_{X}(X)^{\oplus a_{1}}$. From the exactness of $M\mapsto\widetilde{M}$ we obtain an exact sequence
\[
\mathscr{D}_{X}^{\oplus a_{2}}\to\mathscr{D}_{X}^{\oplus a_{1}}\to\widetilde{M}\to0.
\]
It follows from the fully-faithfulness that $\widetilde{M}=\mathscr{M}$ as they both must be the cokernel of the same homomorphism. This finishes the proof of (3).

To prove (4) notice that
\[
\DR_{\cD_{A}}(M)=\Gamma(X,\DR_{X}(\widetilde{M}))
\]
and therefore we only need to check that  for $i>0$ we have $H^{i}(X,\widetilde{M}\otimes_{\cO_{X}}\Omega^{j}_{X})=0.$ Since $\Omega^{j}_{X}$ is globally free of finite rank this follows from (2).
\end{proof}

\subsection{Holonomicity and modules of minimal dimension}\label{Holonomicity}

Let $R$ be a (not necessarily commutative) ring that is both left and right noetherian and of finite global homological dimension. We consider the derived category $D^{b}_{f}(R)$ of bounded complexes of (left) $R$-modules with finitely generated cohomology. The corresponding category of right $R$-modules can be naturally identified with $D^{b}_{f}(R^{op})$. Under our assumptions we have a well defined \textit{duality functor}
\[
\DD_{R}=\RHom_{R}(-,R):D^{b}_{f}(R)\to D^{b}_{f}(R^{op})
\]
Its basic properties are contained in the following lemma.

\begin{lem}\label{BaseChange}
With the notation and under the assumptions above the following hold.
\begin{enumerate}

\item We have an equality $M=\DD_{R^{op}}\DD_{R}(M)$.

\item The duality functor is an equivalence of categories satisfying $\DD_{R^{op}}\DD_{R}=\mathrm{Id}$. 
\item The natural map $\RHom_{R}(M^{\bullet},N^{\bullet})\to\RHom_{R^{op}}(\DD_{R}(N^{\bullet}),\DD_{R}(M^{\bullet}))$ is an isomorphism.
\item
Let $R\to S$ be a ring homomorphism of left and right noetherian rings of finite global homological dimension. Assume that $S$ is flat as a right $R$-module. Then for any $M^{\bullet}\in D^{b}_{f}(R)$ we have $\DD_{S}(S\otimes_{R}M^{\bullet})=\DD_{R}(M^{\bullet})\otimes_{R}S$.
\end{enumerate}
\end{lem}

\begin{proof}
Parts (1)--(3) are well known, see for example \cite[p. 49]{Meb1}, \cite[D.4]{HTT}. Possibly (4) is known to the experts but we were unable to provide a suitable reference, so we give a proof.

If $M$ is a finitely generated (and thus finitely presented by noetherianity) $R$-module then $\Hom_{R}(M,R)\otimes_{R}S=\Hom_{S}(S\otimes_{R}M,S)$ since $S$ is $R$-flat. Under our assumptions every object in $D^{b}_{f}(R)$ is represented by a bounded complex of finitely generated projective modules. If $M^{\bullet}$ is such complex then 
\[
\mathbf{R}\textnormal{Hom}_{R}(M^{\bullet},R)\otimes_{R}S=\mathbf{R}\textnormal{Hom}_{S}(S\otimes_{R}M^{\bullet},S)
\]
and we are done.
\end{proof}

Recall from \cite[Section 1.2]{Meb1} that a finitely generated $R$-module $M$  is \textit{of minimal dimension} if it is either zero or
\[
\textnormal{grade}_{R}(M):=\inf\{i:\textnormal{Ext}^{i}_{R}(M,R)\neq0\}=\textnormal{gl.dim} R.
\]
Let $A$ be a $K$-algebra satisfying properties (i)--(iii) of Lemma \ref{Matsumura}. The category of $\cD_{A}$-modules of minimal dimension is well-understood. If $M$ is a finitely generated left $\cD_{A}$-module then there exists a \textit{good filtration} $F_{\bullet}M$ on $M,$ i.e., a filtration such that $\textnormal{gr}^{F}M$ is a finitely generated $A[\xi_{1},\dots,\xi_{n}]$-module (here $\xi_{i}$ denotes the \textit{symbol} of $\partial_{i}$ ). By \cite[Thm V.2.2.2]{Borel}
\begin{equation}\label{CharacteristicCycle}
\textnormal{grade}_{\cD_{A}}(M)+\dim\textnormal{gr}^{F}M=2n
\end{equation}
where $\dim\textnormal{gr}^{F}M$ is by definition the Krull dimension of $B=\frac{A[\xi_{1},\dots,\xi_{n}]}{\sqrt{\textnormal{Ann}(\textnormal{gr}^{F}M)}}.$ One often refers to the affine scheme $\textnormal{Spec }B$ as the \textit{singular support} or the \textit{characteristic variety} of $M.$\newline

We say that a left $\mathscr{D}_{X}$-module is \textit{holonomic} if there exists an open covering of $X$ by affinoids $\{U_{i}\}$ such that each $U_{i}$ admits a coordinate system and $\mathscr{M}_{|U_{i}}=\widetilde{M}_{i}$ for some left $\cD_{\cO_{X}(U_{i})}$-module $M_{i}$ of minimal dimension.

\begin{lem}\label{HoloMini}
Assume that $X=\textnormal{Spa }A$ admits a global coordinate system. Let $M$ be a finitely generated left $\cD_{A}$-module. Then the coherent $\mathscr{D}_{X}$-module $\mathscr{M}=\widetilde{M}$ is holonomic if and only if $M$ is a $\cD_{A}$-module of minimal dimension.
\end{lem}

\begin{proof}
We need to check that if $\mathscr{M}$ is holonomic then $M$ is of minimal dimension. If $M$ is a finitely generated $\cD_{A}$-module and $U\subset X$ is an affinoid subdomain then from the flatness of $\D(U)$ over $\cD_{A}$ we get (see Lemma \ref{BaseChange} (4))
\begin{equation*}
\textnormal{Ext}^{i}_{\mathscr{D}_{X}(U)}(\widetilde{M}(U),\D(U))=\textnormal{Ext}^{i}_{\cD_{A}}(M,\cD_{A})\otimes_{\cD_{A}}\D.
\end{equation*}
Then an argument with Tate's acyclicity Theorem analogous to the one given in Lemma \ref{AffinoidCorrespondence} shows that the presheaf
\[
\mathscr{N}_{i}:U\mapsto \textnormal{Ext}^{i}_{\cD_{A}}(M,\cD_{A})\otimes_{\cD_{A}}\D(U)
\]
is a sheaf of \textit{right} $\mathscr{D}_{X}$-modules. In particular, we have
\[
\mathscr{N}_{i}(X)=\textnormal{Ext}^{i}_{\cD_{A}}(M,\cD_{A})
\]
Now assume that $\mathscr{M}$ is holonomic and $i\neq\dim X$. Let $\{U_{j}\}$ be the covering of $X$ from the definition of holonomicity. Then 
\[
\mathscr{N}_{i}(U_{j})=\textnormal{Ext}^{i}_{\mathscr{D}_{X}(U_{j})}(\widetilde{M}(U_{j}),\D(U_{j}))=0
\]
and therefore $\mathscr{N}_{i}(X)=0$.
\end{proof}

\begin{ex}\label{Integrability2}
    Let $\mathscr{M}$ be a locally free $\cO_{X}$-module of finite rank and let $\nabla:\mathscr{M}\to\mathscr{M}\otimes_{\cO_{X}}\Omega_{X}$ be an integrable connection in the sense of Remark \ref{Integrability1}. Then the corresponding left $\mathscr{D}_{X}$-module is holonomic. This construction provides a large class of examples of holonomic $\mathscr{D}$-modules that appear naturally in geometry.
\end{ex}

\subsection{Remark about $D$-affinity}\label{Daffinity}

Let $Y=\Spec A$ be a smooth affine variety over $\CC$. Then there is an equivalence of categories between the category of coherent left $\mathscr{D}_{Y}$-modules and the category of finitely generated left $\cD_{A}$-modules and under this equivalence holonomic $\mathscr{D}_{Y}$-modules correspond to modules of minimal dimension. More generally, we say that a smooth algebraic $\CC$-variety $Y$ is $D$-affine if the functor $\Gamma(Y,-)$ is exact and it induces an equivalence of categories between the category of $\mathscr{D}_{Y}$-modules that are $\cO_{Y}$-quasi-coherent and the category of $\mathscr{D}_{Y}(Y)$-modules. It is natural to expect that smooth affinoid varieties are also $D$-affine in the sense that the functor $M\mapsto \widetilde{M}$ of Lemma \ref{AffinoidCorrespondence} gives a desired equivalence of categories. Unfortunately, we do not known if it is true even for the Tate polydiscs. The problem is that, contrary to the situation in classical algebraic geometry, quasi-coherent $\cO_{X}$-modules on affinoid varieties need not be globally generated. If we knew that smooth affinoids are in fact $D$-affine some arguments in the next sections could be simplified.

\subsection{Side-changing operations and direct image along closed embedding}\label{DirectImageSection} Let $X$ be a smooth rigid analytic variety. We write $\omega_{X}=\det\Omega_{X}$ for the canonical line bundle on $X$. It is known (cf. \cite[p. 19]{HTT}) that $\omega_{X}$ is in fact a right $\mathscr{D}_{X}$-module with the differential structure induced by the Lie derivative. More generally, if $\mathscr{M}$ (resp. $\mathscr{M}'$) is a left (resp. right) $\mathscr{D}_{X}$-module then $\omega_{X}\otimes_{\cO_{X}}\mathscr{M}$ (resp. $\mathcal{H}om_{\cO_{X}}(\omega_{X},\mathscr{M}')$) carries a natural structure of a right (resp. left) $\mathscr{D}_{X}$-module and these constructions are inverse to each other (cf. \cite[p. 20]{HTT}). If $x_{1},\dots,x_{n}$ is a coordinate system on $X$, $\partial_{1},\dots,\partial_{n}$ are the corresponding derivations and $dx_{1}\wedge\dots\wedge dx_{n}$ is the corresponding section of $\omega_{X}$ then the passage from left to right $\mathscr{D}_{X}$-modules is given by the involution (\ref{involution}), i.e., by the formula
\begin{equation}\label{local side-change}
dx_{1}\wedge\dots\wedge dx_{n}\otimes m.P=dx_{1}\wedge\dots\wedge dx_{n}\otimes P^{t}m
\end{equation}
Operations described above are called the \textit{side-changing operations}. We need the following easy Lemma.

\begin{lem}\label{LocalSidechange}
Assume that $X$ admits a global coordinate system. Then the right $\mathscr{D}_{X}$-module $\omega_{X}\otimes_{\cO_{X}}\mathscr{D}_{X}$ is free of rank one. Analogous statement holds for the left $\mathscr{D}_{X}$-module $\mathscr{D}_{X}\otimes_{\cO_{X}}\omega_{X}^{\vee}$. 
\end{lem}

\begin{proof}
Let us set $\mathscr{M}=\omega_{X}\otimes_{\cO_{X}}\mathscr{D}_{X}$. We fix a coordinate system $x_{1},\dots,x_{n}$ and let $e=dx_{1}\wedge\dots\wedge dx_{n}\otimes 1$. The action of $\mathscr{D}_{X}$ on $\mathscr{M}$ is defined by action of the tangent sheaf given by
\[
(fdx_{1}\wedge\dots\wedge dx_{n}\otimes Q).\theta=-\theta(f)dx_{1}\wedge\dots\wedge dx_{n}\otimes Q-fdx_{1}\wedge\dots\wedge dx_{n}\otimes \theta Q
\]
We conclude that for $P\in\mathscr{D}_{X}$ we have
\[
(dx_{1}\wedge\dots\wedge dx_{n}\otimes 1).P=dx_{1}\wedge\dots\wedge dx_{n}\otimes P^{t}
\]
and it is easy to see that $\mathscr{M}=e.\mathscr{D}_{X}$ is free of rank one.
\end{proof}

Let $i:X\hookrightarrow Y$ be a Zariski closed embedding of smooth rigid analytic varieties. Let us recall the definition of the \textit{transfer modules}. We have
\[
\mathscr{D}_{X\to Y}=\cO_{X}\otimes_{i^{-1}\cO_{Y}}i^{-1}\mathscr{D}_{Y}=i^{*}\mathscr{D}_{Y}
\]
This is an $(\mathscr{D}_{X},i^{-1}\mathscr{D}_{Y})$-bimodule with the $\mathscr{D}_{X}$-module structure given by the chain rule (cf. \cite[p. 23]{HTT}). We also have
\[
\mathscr{D}_{Y\leftarrow X}=\omega_{X}\otimes_{\cO_{X}}\mathscr{D}_{X\to Y}\otimes_{i^{-1}\cO_{Y}}i^{-1}\omega_{Y}^{\vee}
\]
This is a $(i^{-1}\mathscr{D}_{Y},\mathscr{D}_{X})$-bimodule with the structure induced by the side-changing operations. The \textit{direct image} of a left (resp. right) $\mathscr{D}_{X}$-module is defined as
\[
i_{+}\mathscr{M}=i_{*}(\mathscr{D}_{Y\leftarrow X}\otimes_{\mathscr{D}_{X}}\mathscr{M})
\]
(resp. $i_{*}(\mathscr{M}\otimes_{\mathscr{D}_{X}}\mathscr{D}_{X\to Y})$). While our notation is widely used, another common notation for the direct image, used for example in \cite{Meb1}, \cite{HTT}, is $\int_{i}\mathscr{M}$.

It is a standard computation (cf. \cite[p. 24]{HTT}) that (for a closed embedding) the transfer modules are locally free over $\mathscr{D}_{X}$. Moreover, since $i$ is affine $i_{*}$ is an exact functor. It follows that $i_{+}$ is an exact functor. The same computation shows that if $y_{1},\dots,y_{n}$ is a coordinate system for the closed embedding $i:X\to Y$ such that $X$ is cut out by the ideal $\cI=(y_{r+1},\dots,y_{n})$ then
\begin{equation}\label{LocalDirectImage}
i_{+}\mathscr{M}=\bigoplus_{i_{r+1},\dots,i_{n}}i_{*}\mathscr{M}.\partial_{r+1}^{i_{r+1}}\dots\partial_{n}^{i_{n}}=i_{*}\mathscr{M}[\partial_{r+1},\dots,\partial_{r}].
\end{equation}
If moreover $Y=\textnormal{Spa }A$ and $X=\textnormal{Spa }B$ with $B=A/\cI$ then the choice of the coordinate system induces the homomorphism
\[
\cD_{A}=A[\partial_{1},\dots,\partial_{n}]\to B[\partial_{1},\dots,\partial_{n}]=\cD_{B}[\partial_{r+1},\dots,\partial_{n}]
\]
and the $\mathscr{D}_{Y}$-module structure on (\ref{LocalDirectImage}) is the one induced by this homomorphism. Finally, we mention that the formation of direct images commutes with the side-changing operations in the sense that we have a commutative diagram (cf. \cite[p. 23]{HTT})
\begin{equation}\label{Left-Right}
\begin{tikzcd}
\textnormal{Mod}(\mathscr{D}_{X})\arrow{r}{\omega_{X}\otimes-}\arrow{d}{i_{+}}&\textnormal{Mod}(\mathscr{D}_{X}^{op})\arrow{d}{i_{+}}\\
\textnormal{Mod}(\mathscr{D}_{Y})\arrow{r}{\omega_{Y}\otimes-}&\textnormal{Mod}(\mathscr{D}_{Y}^{op})
\end{tikzcd}
\end{equation}

\begin{rmk}
In this subsection we referred mostly to the definitions and basic properties of $\mathscr{D}$-modules contained in the first chapter of the book \cite{HTT} of Hotta--Takeuchi--Tanisaki. This book deals with $\mathscr{D}$-modules on complex algebraic varieties but it is clear that all the stated facts translate \textit{mutatis mutandis} to our setting as they are formal consequences of the fact that $\mathscr{D}_{X}$ is generated by the tangent bundle. In the following section we need to be more careful.
\end{rmk}

\subsection{Cohomological descent} To reduce the proof of Theorem \ref{MainThm} to the affinoid case we need to use some spectral sequences. Let $X$ be a smooth rigid analytic variety and let $X=\bigcup_{i=1}^{N}U_{i}$ be an open cover. We define
\[
\mathscr{U}^{0}=\coprod_{i} U_{i}
\]
and
\[
\mathscr{U}^{n}=\underbrace{\mathscr{U}^{0}\times_{X}\mathscr{U}^{0}\times_{X}\dots\times_{X}\mathscr{U}^{0}}_{n+1\textnormal{-times}}
\]
In other words $\mathscr{U}^{n}$ are disjoint sums of the intersections of $(n+1)$ elements from the open cover of $X$. We refer to \cite[5.3.3.3]{Deligne} for the proof of the following lemma.

\begin{lem}\label{SpectralSequence}
Let $\cF^{\bullet}$ be a complex of sheaves of abelian groups on $X$. Then there exists a spectral sequence
\[
E_{1}^{p,q}=H^{q}(\mathscr{U}^{p},\cF^{\bullet}_{|\mathscr{U}^{p}})\implies H^{p+q}(X,\cF^{\bullet})
\]
\end{lem}

Obviously, if $\cF^{\bullet}$ in a complex sheaves of $K$-vector spaces then the maps in the spectral sequence are $K$-linear. By taking $\cF^{\bullet}$ the de Rham complex of some $\mathscr{D}_{X}$-module we obtain the following.

\begin{lem}\label{deRhamSpectralSequence}
Let $X=\bigcup_{i=1}^{N}U_{i}$ and let $\mathscr{M}$ be a coherent left $\mathscr{D}_{X}$-module. Then there exists a spectral sequence of $K$-vector spaces
\[
E_{1}^{p,q}=\bigoplus_{1\leq i_{1},\dots,i_{p}\leq N}H^{q}_{\mathrm{dR}}(U_{i_{1}}\cap\dots\cap U_{i_{p}},\mathscr{M}_{|U_{i_{1}}\cap\dots\cap U_{i_{p}}})\implies H^{p+q}_{\mathrm{dR}}(X,\mathscr{M})
\]
\end{lem}

\begin{proof}
This is a direct consequence of Lemma \ref{SpectralSequence} and the definition of the de Rham cohomology.
\end{proof}

\subsection{Other base fields}

Everything discussed so far holds if we replace $K$ by any nonarchimedean field of characteristic zero. From the next section on, the assumption that $K$ is of equal characteristic zero will be crucial and the corresponding results will not hold in general in mixed characteristic. The assumption that $K$ is discretely valued is used (more or less explicitly) several times but we expect that with some extra work it can be removed.

\section{Globally presented $\mathscr{D}$-modules on Tate polydiscs}\label{Polydiscs}
The goal of this section is to verify Theorem \ref{MainThm} in case $X=\BB^{n}$ and $\mM$ is a globally presented holonomic $\D$-module.

\begin{prop}\label{TateDisc}
Let $X=\BB^{n}$ and let $\mathscr{M}$ be a globally presented holonomic $\mathscr{D}_{X}$-module. Then $\dim_{K}H_{\mathrm{dR}}^{i}(X,\mathscr{M})<\infty$ for all $i$.
\end{prop}

This is done by reducing the problem to the de Rham cohomology of modules of minimal dimensions over $\mathscr{D}_{X}(X)$ and further to the modules of minimal dimension over completed Weyl algebras.

\subsection{Completed Weyl Algebras}

Let us set
\[
\cD_{n}=\D(X)=K\langle x_{1},\dots,x_{n}\rangle[\partial_{1},\dots,\partial_{n}].
\] 
The second equality holds by Lemma \ref{Matsumura}. We also write
\[\WW_{n}(\oo)=\oo[x_{1},\dots,x_{n}][\partial_{1},\dots,\partial_{n}]
\]
for the $n$-th Weyl algebra over $\oo$. Recall that $A=K\langle x_{1},\dots,x_{n}\rangle$ is a Banach algebra with respect to the \textit{Gauss norm} $|\sum a_{\alpha}x^{\alpha}|=\max|a_{\alpha}|.$ The natural action of $\cD_{n}$ on $A$ is continuous and therefore the Gauss norm induces an operator norm on $\cD_{n}.$ One checks that
\[
|\sum f_{\alpha}\partial^{\alpha}|=\max|f_{\alpha}|.
\]
We define $\cD_{n}^{\circ}=\{P\in\cD_{n}:|P|\leq1\}.$ Then  $\cD_{n}^{\circ}$ is an $\oo$-algebra such that $\cD^{\circ}_{n}\otimes_{\oo}K=\cD_{n}.$ Moreover, we have
\[
\frac{\cD_{n}^{\circ}}{\varpi^{k+1}{\cD_{n}^{\circ}}}=\WW_{n}(\oo/\varpi^{k+1}\oo)=\frac{\WW_{n}(\oo)}{\varpi^{k+1}\WW_{n}(\oo)}
\]
and thus
\[
\varprojlim\frac{\cD_{n}^{\circ}}{\varpi^{k+1}{\cD_{n}^{\circ}}}=\varprojlim\frac{\WW_{n}(\oo)}{\varpi^{k+1}\WW_{n}(\oo)}=\widehat{\cD}_{n}^{\circ}
\]
We set $\Dd=\widehat{\cD}_{n}^{\circ}\otimes_{\oo}K$. This is the \textit{completed Weyl algebra} over $K$. It is a left and right noetherian algebra of global homological dimension $n$ (see \cite[Lemma 2.3]{Felek}). We have natural injective maps $\cD_{n}^{\circ}\to\widehat{\cD}_{n}^{\circ}$ and $\cD_{n}\to\Dd.$
\begin{rmk}
Completed Weyl algebras has been studied by many mathematicians, notably by L.\ N{\'a}rvaez Macarro in \cite{Narvaez} and more recently by A.\ Pangalos in his PhD thesis \cite{Pangalos}. Our notation $\Dd$ should not be confused with the ring $\wideparen{D}$ (pronounced `D-cap') of Ardakov--Bode--Wadsley (cf. \cite{Ardakov1}, \cite{Ardakov2}, \cite{Ardakov3}). Since the use of `hat' to denote a completion of a ring is widely used we believe that our notation is justified.
\end{rmk}
If $M$ is a left $\Dd$-module then it is also a $\cD_{n}$-module and as such it has a the Rham complex and de Rham cohomology defined by formulas (\ref{de Rham complex 2}). In our earlier work \cite{Felek} we have studied the de Rham cohomology of $\Dd$-modules of minimal dimension (cf. Subsection \ref{Holonomicity}) and we have obtained the following result.
\begin{thm}[{\cite[Theorem 1.1]{Felek}}]\label{MainThm0}
Let $K$ be a discretely valued nonarchimedean field of equal characteristic zero and let $M$ be a left $\Dd$-module of minimal dimension. Then $\dim_{K}H^{i}_{\mathrm{dR}}(M)<\infty$ for all $i$.
\end{thm}
Our proof of Proposition \ref{TateDisc} is ultimately reduced to the above theorem.

\subsection{The base change $\cD_{n}\to\Dd$} Now we study which properties of $\cD_{n}$-modules are preserved after tensoring with $\Dd$. If $M$ is a left $\cD_{n}$-module we write $\widehat{M}=\Dd\otimes_{\cD_{n}}M$.

\begin{lem}\label{flat}
$\Dd$ is flat as a left and right $\cD_{n}$-module.
\end{lem}

\begin{proof}
We first check that $\cD_{n}^{\circ}$ is left and right noetherian. By \cite[D.1.4]{HTT} it is sufficient to find a filtration on $\cD_{n}^{\circ}$ such that the associated graded is noetherian. Setting
\[
F_{t}\cD_{n}^{\circ} =\bigoplus_{|\alpha|\leq t}\oo\langle x_{1},\dots,x_{n}\rangle \partial^{\alpha}
\]
it is easy to see that this is indeed a filtration in the sense of \cite[ Appendix D]{HTT} and that the associated graded is a (commutative) polynomial ring in $\xi_{1},\dots,\xi_{n}$, i.e., 
\[
\mathrm{gr}^{F_{\bullet}}\cD_{n}^{\circ}=\oo\langle x_{1},\dots,x_{n}\rangle[\xi_{1},\dots,\xi_{n}].
\]
It is noetherian by Hilbert's basis theorem because $\oo\langle x_{1},\dots,x_{n}\rangle$ is noetherian if $\oo$ is a discrete valuation ring.

Now it is well known that if $R$ is a commutative noetherian ring then the $I$-adic completion $\widehat{R}^{I}$ is $R$-flat for all ideals $I\subset R$. The proof presented in \cite[Chapter 10]{Atiyah}  is easy to generalize to the case when $R$ is noncommutative assuming that Artin--Rees lemma holds for the $I$-adic topology on $R$. This assumption is satisfied if $I$ is generated by a central element by \cite[p. 413]{Row}. Taking $I=(\varpi)\subset\cD_{n}^{\circ}$ we see that $\widehat{\cD}_{n}^{\circ}$ is flat over $\cD_{n}^{\circ}.$ Then $\Dd=\widehat{\cD}_{n}^{\circ}[\varpi^{-1}]$ is flat over $\cD_{n}=\cD_{n}^{\circ}[\varpi^{-1}]$ because flatness is preserved under localization by \cite[Prop 3.2.9]{Weibel}.
\end{proof}

\begin{lem}\label{CompletedWeyl}
Let $M$ be a finitely generated left $\cD_{n}$-module.

\begin{enumerate}

\item If $M$ is of minimal dimension then so is $\widehat{M}$.

\item The complexes $\DR_{\cD_{n}}(M)$ and $\DR_{\Dd}(\widehat{M})$ are quasi-isomorphic.

\end{enumerate}

\end{lem}

\begin{proof}

Since $\Dd$ is flat over $\cD_{n}$ we have
\[
\textnormal{Ext}^{i}_{\Dd}(\widehat{M},\Dd)=\textnormal{Ext}^{i}_{\cD_{n}}(M,\cD_{n})\otimes_{\cD_{n}}\Dd
\]
by the Lemma \ref{BaseChange} (4) and since $\cD_{n}$ and $\Dd$ both have homological dimension $n$, assertion  (1) holds. It is less obvious why (2) holds. To prove it we use multiple times equality
\begin{equation}\label{BC1}
\DD_{\Dd}(\widehat{M})=\DD_{\cD_{n}}(M)\otimes_{\cD_{n}}\Dd,
\end{equation}
which holds by (4) of Lemma \ref{BaseChange}. First, let us consider consider the left $\cD_{n}$-module 
\[
\cO_{X}=\frac{\cD_{n}}{\cD_{n}(\partial_{1},\dots,\partial_{n})}
\]
and the left $\Dd$-module 
\[
\widehat{\cO}_{X}=\frac{\Dd}{\Dd(\partial_{1},\dots,\partial_{n})}=\Dd\otimes_{\cD_{n}}\cO_{X}.
\]
Note that $\cO_{X}=\widehat{\cO}_{X}$ as left $\cD_{n}$-modules as the former is constructed by forgetting the $\Dd$-module structure on the latter. We prefer to distinguish between these two objects for the clarity of the proof.

It follows from (\ref{SpencerDR}) that we have an equality
\begin{equation}\label{RHom1}
\DR_{\cD_{n}}(M)=\RHom_{\cD_{n}}(\cO_{X},M).
\end{equation}
Then we have
\begin{equation}\label{RHom2}
\DR_{\Dd}(\widehat{M})
=\DR_{\cD_{n}}(\widehat{M})
=\RHom_{\cD_{n}}(\cO_{X},\widehat{M})
=\RHom_{\Dd}(\widehat{\cO}_{X},\widehat{M})
\end{equation}
where the first equality is the definition, the second is (\ref{RHom1}) and the last one follows from the left adjointness of the tensor product to the restriction of scalars.
By a standard computation with the Spencer complex (cf. \cite[Proof of Proposition 2.6.12]{HTT}) we have equality
\begin{equation*}
\DD_{\cD_{n}}(\cO_{X})=\frac{\cD_{n}}{(\partial_{1},\dots,\partial_{n})\cD_{n}}[-n]
\end{equation*}
and thus also
\begin{equation*}
\DD_{\Dd}(\widehat{\cO}_{X})=\DD_{\cD_{n}}(\cO_{X})\otimes_{\cD_{n}}\Dd=\frac{\Dd}{(\partial_{1},\dots,\partial_{n})\Dd}[-n].
\end{equation*}
Clearly, the natural map 
\[
\frac{\cD_{n}}{(\partial_{1},\dots,\partial_{n})\cD_{n}}\to\frac{\Dd}{(\partial_{1},\dots,\partial_{n})\Dd}
\]
is an isomorphism of right $\cD_{n}$-modules and therefore 
\begin{equation}\label{Dual3}
\DD_{\cD_{n}}(\cO_{X})=\DD_{\Dd}(\widehat{\cO}_{X}).
\end{equation}
We now explain why the following chain of equalities is true. 
\begin{equation}
\begin{split}
\DR_{\Dd}(\widehat{M})
&\stackrel{\ref{RHom2}}=\RHom_{\Dd}(\widehat{\cO}_{X},\widehat{M})\\
&=\RHom_{\Dd^{op}}(\DD_{\Dd}(\widehat{M}),\DD_{\Dd}(\widehat{\cO}_{n}))\\
&\stackrel{\ref{BC1}}=\RHom_{\Dd^{op}}(\DD_{\cD_{n}}(M)\otimes_{\cD_{n}}\Dd,\DD_{\Dd}(\widehat{\cO}_{X}))\\
&=\RHom_{\cD_{n}^{op}}(\DD_{\cD_{n}}(M),\DD_{\Dd}(\widehat{\cO}_{X}))\\
&\stackrel{\ref{Dual3}}=\RHom_{\cD_{n}^{op}}(\DD_{\cD_{n}}(M),\DD_{\cD_{n}}(\cO_{X}))\\
&=\RHom_{\cD_{n}}(\cO_{X},M)\\
&\stackrel{\ref{RHom1}}=\DR_{\cD_{n}}(M)
\end{split}
\end{equation}
For the first and the last equality without a subscript we use part (3) of Lemma \ref{BaseChange}. For the remaining equality we use that $(-)\otimes_{\cD_{n}}\Dd$ is left adjoint to forgetting the structure from $\Dd$ to $\cD_{n}.$ This ends the proof of the lemma.
\end{proof}

\subsection{Proof of Proposition \ref{TateDisc}}

We finish this section with the proof of Proposition \ref{TateDisc}.

\begin{proof}
Let $\mathscr{M}$ be a globally finitely presented holonomic $\mathscr{D}_{X}$ module (in this section $X=\BB^{n}$) and let $M$ be the $\cD_{n}$-module that corresponds to $\mathscr{M}$ by Lemma \ref{AffinoidCorrespondence}. Then by Lemma \ref{HoloMini} $M$ is of minimal dimension and we have
\[
H_{\mathrm{dR}}^{i}(X,\mathscr{M})=H^{i}_{\mathrm{dR}}(M).
\]
On the other, hand by Lemma \ref{CompletedWeyl} we know that $\widehat{M}$ is of minimal dimension and 
\[
H^{i}_{\mathrm{dR}}(M)=H^{i}_{\mathrm{dR}}(\widehat{M}).
\]
Therefore the proposition follows from Theorem \ref{MainThm0}.
\end{proof}

\begin{ex}\label{Nonholonomic}
Consider the left $\cD_{n}$-module $M=\frac{\cD_{n}}{\cD_{n}(1-\varpi\partial_{1})}.$ It is easy to see that it is nonzero and it is not of minimal dimension for $n\geq2.$ Note that $1-\varpi\partial_{1}$ is a unit in $\Dd$ with the inverse $(1-\varpi\partial_{1})^{-1}=\sum_{k\geq0}\varpi^{k}\partial_{1}^{k}$ and therefore $\widehat{M}=\Dd\otimes_{\cD_{n}}M=0.$ This shows that $\Dd$ is not faithfully flat over $\cD_{n}$ and gives an example of a left $\cD_{n}$-module which is not of minimal dimension but (by Lemma \ref{CompletedWeyl}) has finite de Rham cohomology. This example shows also that the converse to the Lemma \ref{CompletedWeyl} (1) cannot hold.
\end{ex}

\section{$\mathscr{D}$-module theoretic direct image along a closed embedding}

In this section we investigate how holonomicity and finiteness of the de Rham cohomology behave under direct images along Zariski closed embeddings to conclude the proof of the following proposition.

\begin{prop}\label{GlobalCoordinates}
Let $X$ be a smooth (affinoid) rigid analytic variety that admits a global coordinate system and let $\mathscr{M}$ be a globally finitely presented (left) holonomic $\mathscr{D}_{X}$-module. Then $\dim_{K}H^{i}_{\mathrm{dR}}(X,\mathscr{M})<\infty$ for all $i$.
\end{prop}

\subsection{Properties of the direct image}
Most of this subsection is occupied by the proof of the following lemma, which collects properties of $i_{+}$ (see Subsection \ref{DirectImageSection}) needed for the proof of Theorem \ref{MainThm}.

\begin{lem}\label{DirectImageLemma}
  Let $i:X\hookrightarrow Y$ be a Zariski closed embedding of smooth rigid analytic varieties. Let $\mathscr{M}$ be a coherent left $\mathscr{D}_{X}$-module. Then
  \begin{enumerate}
\item If both $X$ and $Y$ admit global coordinate systems and $\mathscr{M}$ is globally finitely presented then so is $i_{+}\mathscr{M}$.

\item The left $\mathscr{D}_{Y}$-module $i_{+}\mathscr{M}$ is coherent.

\item If $\mathscr{M}$ is holonomic then so is $i_{+}\mathscr{M}$.

\item There exists a natural $K$-linear quasi-isomorphism of complexes
\[
i_{*}\DR_{X}(\mathscr{M})\to\DR_{Y}(i_{+}\mathscr{M})[\dim X-\dim Y].
\]
  \end{enumerate}
\end{lem}

\begin{proof}
Let $\mathscr{M}$ be a globally finitely presented left $\mathscr{D}_{X}$-module (resp. globally finitely presented right $\mathscr{D}_{Y}$-module). Since the functor $\omega_{X}\otimes_{\cO_{X}}-$ (resp. $-\otimes_{\cO_{Y}}\omega_{Y}^{\vee}$) is clearly exact Lemma \ref{LocalSidechange} implies that it preserves finite global presentation. It follows from the commutativity of the diagram (\ref{Left-Right}) and the discussion above that the left $\mathscr{D}_{Y}$-module $i_{+}\mathscr{M}$ is globally finitely presented if the right $\mathscr{D}_{Y}$-module $i_{+}(\omega_{X}\otimes_{\cO_{X}}\mathscr{M})$ is globally finitely presented. On the other hand, we have already remarked that the functor $i_{+}$ is exact. Therefore it suffices to show that the right $\mathscr{D}_{Y}$-module $i_{+}\mathscr{D}_{X}$ is globally finitely presented. However, if $\mathscr{I}$ is the ideal that cuts out $X$ inside $Y$ then
\[
i_{+}\mathscr{D}_{X}=i_{*}\mathscr{D}_{X\to Y}=i_{*}i^{*}\mathscr{D}_{Y}=\mathscr{D}_{Y}/\mathscr{I}\mathscr{D}_{Y}.
\]
The right-hand side is clearly globally finitely presented as $\mathscr{I}$ is generated by finitely many elements. This establishes (1). (2) follows easily from (1) after passing to the local coordinates for the embedding $i:X\hookrightarrow Y$.
\newline

Statement (3) is local so we may assume that the embedding $i:X=\textnormal{Spa }B\hookrightarrow\textnormal{Spa }A=Y$ admits a global coordinate system $y_{1},\dots,y_{n}$ with $X=\{y_{r+1}=\dots=y_{n}=0\}$ and $\mathscr{M}$ is globally finitely presented. Write $\partial_{i}$ for the derivation dual to $dy_{i}.$ Then $\cD_{A/K}=A[\partial_{1},\dots,\partial_{n}]$, $\cD_{B/K}=B[\partial_{1},\dots,\partial_{r}]$ and if $\mM$ corresponds to a $\cD_{B/K}$-module $M$ (see Lemma \ref{AffinoidCorrespondence}) then by formula (\ref{LocalDirectImage}) the direct image $i_{+}\mM$ corresponds to $M'=M[\partial_{n+1},\dots,\partial_{m}].$ We have to show that if $M$ was of minimal dimension then so is $M'$. For that we use the characterisation of holonomicity in terms of the dimension of the characteristic variety (see Subsection \ref{Holonomicity}). By equality (\ref{CharacteristicCycle}) if $F_{\bullet}M$ is a good filtration on $M$ then $\dim\mathrm{gr}^{F}M=r$ and we want to show that for some good filtration $G_{\bullet}$ on $M'$ we have $\dim \mathrm{gr}^{G}M'=n.$ Set
\[
G_{t}M'=\bigoplus_{0\leq|\alpha|\leq t}F_{t-|\alpha|}M.\partial^{\alpha},
\]
where $\partial^{\alpha}=\partial_{r+1}^{\alpha_{1}}\dots\partial_{n}^{\alpha_{m-r}}$ and $|\alpha|=\sum\alpha_{i}$. Then
\[
\mathrm{gr}_{t}^{G}M'=\bigoplus_{0\leq|\alpha|\leq t}\mathrm{gr}^{F}_{t-|\alpha|}M.\xi^{\alpha}
\]
and therefore
\[
\mathrm{gr}^{G}M'=\bigoplus_{t\geq0}\mathrm{gr}_{t}^{G}M=(gr^{F}M)[\xi_{r+1},\dots,\xi_{n}].
\]
In particular, we have
\[
\dim\mathrm{gr}^{G}M'=\dim (\mathrm{gr}^{F}M)[\xi_{r+1},\dots,\xi_{n}]=r+(n-r)=n,
\]
which proves (3).\newline

The proof of (4) which is the most complicated. Write $c=\dim Y-\dim X$. First we describe the natural $\cO_{Y}$-linear maps
\[
f^{i}:i_{*}(\mathscr{M}\otimes_{\cO_{X}}\Omega^{i}_X)\to i_{+}\mathscr{M}\otimes_{\cO_{Y}}\Omega^{c+i}_{Y}
\]
and then we check that these maps give the desired quasi-isomorphism $i_{*}\DR_{X}(\mathscr{M})\to\DR_{Y}(i_{+}\mathscr{M})[-c]$ by computations in the local coordinates for the closed embedding. Recall that we have the conormal short exact sequence
\[
0\to\mathscr{N}_{X/Y}^{\vee}\to\Omega^{1}_{Y|X}\to\Omega^{1}_{X}\to 0,
\]
which induces exact sequences
\begin{equation}\label{di1}
\Omega^{i-1}_{Y|X}\otimes_{\cO_{X}}\mathscr{N}_{X/Y}\to\Omega^{i}_{Y|X}\to\Omega^{i}_{X}\to 0
\end{equation}
and the natural isomorphisms
\begin{equation}\label{di2}
\bigwedge^{c}\mathscr{N}_{X/Y}=\det\mathscr{N}_{X/Y}=\omega_{X}\otimes_{\cO_{X}}\omega_{Y|X}^{\vee}.
\end{equation}
It follows from (\ref{di1}) that the natural pairing $\Omega^{i}_{Y|X}\otimes_{\cO_{X}}\det\mathscr{N}^{\vee}_{X/Y}\to\Omega_{Y|X}^{c+i}$ induces a natural pairing
\begin{equation}\label{di3}
\Omega^{i}_{X}\otimes_{\cO_{X}}\det\mathscr{N}^{\vee}_{X/Y}\to\Omega_{Y|X}^{c+i}.
\end{equation}
Now recall from subsection \ref{DirectImageSection} that we have an isomorphism of $\cO_{X}$-modules
\[
\mathscr{D}_{Y\leftarrow X}=\omega_{X}\otimes_{\cO_{X}}(\cO_{X}\otimes_{i^{-1}\cO_{Y}}i^{-1}\mathscr{D}_{Y}\otimes_{i^{-1}\cO_{Y}}i^{-1}\omega_{Y}^{\vee})
\]
From this description we easily construct a natural $\cO_{X}$-linear map
\[
\det\mathscr{N}_{X/Y}=\omega_{X}\otimes_{\cO_{X}}(\cO_{X}\otimes_{i^{-1}\cO_{Y}} i^{-1}\omega_{Y}^{\vee})\to\mathscr{D}_{Y\leftarrow X}
\]
given locally (in a coordinate system) as
\[
\partial_{r+1}\wedge\dots\wedge\partial_{n}\mapsto dx_{1}\wedge\dots\wedge dx_{r}\otimes 1\otimes 1\otimes \partial_{1}\wedge\dots\wedge\partial_{n}.
\]
After tensoring with $\mathscr{M}$ we obtain an $\cO_{X}$-linear map
\begin{equation}\label{di4}
\mathscr{M}\otimes_{\cO_{X}}\det\mathscr{N}_{X/Y}\to\mathscr{D}_{Y\leftarrow X}\otimes_{\mathscr{D}_{X}}\mathscr{M}
\end{equation}
given locally as
\[
m\otimes\alpha\mapsto\alpha\otimes1\otimes m.
\]
Now (\ref{di3}) and (\ref{di4}) give natural maps
\[
\mathscr{M}\otimes_{\cO_{X}}\Omega^{i}_{X}\to (\mathscr{D}_{Y\leftarrow X}\otimes_{\mathscr{D}_{X}}\mathscr{M})\otimes_{\cO_{X}}\Omega^{i}_{X}\otimes_{\cO_{X}}\det\mathscr{N}^{\vee}_{X/Y}\to(\mathscr{D}_{Y\leftarrow X}\otimes_{\mathscr{D}_{X}}\mathscr{M})\otimes_{\cO_{X}}\Omega^{c+i}_{Y|X}
\]
Finally, by applying $i_{*}$ and using the projection formula we obtain morphisms
\[
f^{i}:i_{*}(\mathscr{M}\otimes_{\cO_{X}}\Omega^{i}_{X})\to i_{+}\mathscr{M}\otimes_{\cO_{Y}}\Omega^{c+i}_{Y}.
\]

We now check that $f^{\bullet}$ defines the desired quasi-isomorphism of de Rham complexes. We write $\delta^{i}$ for the differential in the de Rham complex First, we need to verify that $f^{i+1}\delta^{i}=\delta^{i}f^{i+1}$ (i.e., that $f^{\bullet}$ is a morphism of complexes). Then we show that $f^{\bullet}$ is injective and the cokernel of $f^{\bullet}$ is acyclic. All these questions are local so we may work under the assumptions and using the notation made in the proof of (3). A choice of local coordinates $y_{1},\dots,y_{n}$ for the embedding $i:X\hookrightarrow Y$ induces bases $\{dy_{1},\dots,dy_{n}\}$, $\{dy_{1},\dots,dy_{r}\}$, and $\eta=dy_{r+1}\wedge\dots\wedge dy_{n}$ of $\Omega^{1}_{Y|X}$, $\Omega^{1}_{X}$ and $\mathscr{N}_{X/Y}^{\vee}$ respectively. Under the identifications $\mathscr{M}=\widetilde{M}$ and $i_{+}\mathscr{M}=i_{*}\mathscr{M}[\partial_{r+1},\dots,\partial_{n}]$, the maps $f^{i}$ correspond to the module homomorphism
\[
\bigoplus_{|I|=i}M.dy_{I}\to\bigoplus_{|J|=c+i}M[\partial_{r+1},\dots,\partial_{n}].dy_{J};\quad \quad\alpha\mapsto\alpha\wedge\eta,
\]
where $I\subset \{1,\dots,r\}$ and $J\subset\{1,\dots,n\}$. From that we easily see that
\begin{align*}
\left(f^{i+1}\delta^{i}-\delta^{i}f^{i}\right)(m.dy_{I})
&=f^{i+1}\left(\sum_{j=1}^{r}\partial^{j}m.dy_{j}\wedge dy_{I}\right)-\delta^{i}(m.dy_{I}\wedge\eta)\\
=\sum_{j=1}^{r}\partial^{j}m.dy_{j}\wedge dy_{I}\wedge\eta&-\sum_{j=1}^{r}\partial^{j}m.dy_{j}\wedge dy_{I}\wedge\eta-\sum_{j=r+1}^{n}\partial^{j}m.dy_{j}\wedge dy_{I}\wedge\eta\\
&=-\sum_{j=r+1}^{n}\partial^{j}m.dy_{j}\wedge dy_{I}\wedge\eta=0,
\end{align*}
where the last equality follows simply from the fact that $dy_{j}\wedge\eta=0$ for $j\geq r+1$. This shows that $f^{\bullet}$ is in fact a morphism of complexes.

We now show that $f^{\bullet}$ is a quasi-isomorphism. Clearly, it is injective. It is also clear from the local descriptions of $i_{+}\mathscr{M}$ and $f^{i}$ that we can prove our statement by induction on $c$ so we can assume that $c=1$. Let $K^{\bullet}=\textnormal{coker }f^{\bullet}$. Let us show that the identity map $K^{\bullet}\to K^{\bullet}$ is chain-homotopic to zero. We have
\[
K^{t}=\bigoplus_{I=(1\leq i_{1}<\dots<i_{t}\leq n-1)}M[\partial_{n}].dy_{I}\oplus\bigoplus_{J=(1\leq j_{1}\leq\dots\leq j_{t-1}\leq n-1)}\partial_{n}M[\partial_{n}].dy_{J}
\]
so that every element in $K^{t}$ can be uniquely represented as  $\alpha+\partial_{n}\beta\wedge dy_{n}$ where $\alpha$ (resp. $\beta$) is a $M[\partial_{n}]$-valued $t$-form (resp. $(t-1)$-form) that does not contain $dy_{n}$. We define the homotopy operators $h^{t}:K^{t}\to K^{t-1}$ by
\[
h^{t}:\alpha+\partial_{n}\beta\wedge dy_{n}\mapsto (-1)^{t+1}\beta.
\]
Note that these maps are well-defined because $\partial_{n}$ is not a zero-divisor on $M[\partial_{n}]$. We have to verify the identity
\[
\delta^{t-1}h^{t}+h^{t+1}\delta^{t}=\mathrm{Id}.
\]
We have
    \begin{align*}
       &(\delta^{t-1}h^{t}+h^{t+1}\delta^{t})(\alpha+\partial_{n}\beta\wedge dy_{n})=\\
&\qquad(-1)^{t+1}\delta^{t-1}(\beta)+h^{t+1}\left(\sum_{j=1}^{n}dy_{j}\wedge\partial^{j}\alpha+\sum_{j=1}^{n-1}dy_{i}\wedge \partial_{j}\partial_{n}\beta\wedge dy_{n}\right)=\\
&\qquad(-1)^{t+1}\sum_{j=1}^{n}dy_{j}\wedge \partial_{j}\beta+\alpha+(-1)^{t+2}\sum_{j=1}^{n-1}dy_{j}\wedge \partial_{j}\beta=\\
&\qquad\alpha+(-1)^{t+1}dy_{n}\wedge\partial_{n}\beta=\\
&\qquad\alpha+\partial_{n}\beta\wedge dy_{n}.
    \end{align*}
This concludes the proof of (4).
\end{proof}

\begin{rmk}
It is well known that the analogue of the above lemma holds also for algebraic $\mathscr{D}$-modules. While it is possible that the usual proof carries over to our situation, it would require a lot of space to honestly verify that and therefore we prefer to give an elementary argument. While we are working in the rigid analytic setting, it is clear that our argument translates to any other reasonable category of $\mathscr{D}$-modules. The direct description of the quasi-isomorphism $i_{*}\DR_{X}(\mathscr{M})\to\DR_{Y}(i_{+}\mathscr{M})[\dim X-\dim Y]$ given above does not seem to appear in the standard literature on the subject.
\end{rmk}

\begin{rmk}
In the first part of Lemma \ref{DirectImageLemma} we do not assume that the embedding $X\hookrightarrow Y$ admits a global coordinate system but only that both $X$ and $Y$ do. If $X=\textnormal{Spa }A$ then we may (by the definition of an affinoid variety) write $A=K\langle y_{1},\dots,y_{n}\rangle/I$. Such a choice of a presentation induces a closed embedding $X\hookrightarrow \BB^{n}$. If $X$ admits a global coordinate system then the lemma applies to this embedding although it need not be the case that $X$ is globally cutout in $\BB^{n}$ by some of the coordinates $y_{1},\dots,y_{n}$.
\end{rmk}

\begin{rmk}
The above proof implies that for a \textit{right} globally finitely presented $\mathscr{D}_{X}$-module $\mathscr{M}$, the direct image $i_{+}\mathscr{M}$ is always globally finitely presented. It is not clear if this is also true for left $\mathscr{D}_{X}$-modules when we drop the assumptions about the coordinate systems. The problem is that in general the side-changing operations used in the proof do not preserve global generation. For example, on the projective space the left $\mathscr{D}_{\PP^{n}}$-module $\cO_{\PP^{n}}$ is globally finitely presented but the corresponding right $\mathscr{D}_{\PP^{n}}$-module $\omega_{\PP^{n}}=\cO_{\PP^{n}}(-n-1)$ has no nonzero global sections.
\end{rmk}

\begin{rmk}
Note that in this article we study direct image only  in two special cases. The first one is the case of the closed embedding discussed in this section. The second one is studied more implicitly. The de Rham cohomology of a $\mathscr{D}$-module is the $\mathscr{D}$-module theoretic direct image to the point along the structure morphism. In the theory of algebraic $\mathscr{D}$-modules one studies direct images in greater generality but we do not expect similar theory to work in the rigid analytic case.
\end{rmk}

\subsection{Proof of Proposition \ref{GlobalCoordinates}}

Now we deduce Proposition \ref{GlobalCoordinates} from Proposition \ref{TateDisc} and Lemma \ref{DirectImageLemma}.

\begin{proof}[Proof of Proposition \ref{GlobalCoordinates}]
Let $X$ be a smooth affinoid variety that admits a global coordinate system. There exists a closed embedding $i:X\hookrightarrow \BB^{n}$ for some $n$. Let $\mathscr{M}$ be a globally presented holonomic left $\mathscr{D}_{X}$-module. Since both $X$ and $\BB^{n}$ admit global coordinate systems it follows from part (1) of Lemma \ref{DirectImageLemma} that $i_{+}\mathscr{M}$ is globally finitely presented. It is also holonomic by (3) of the same lemma. From Proposition \ref{TateDisc} we conclude that $i_{+}\mathscr{M}$ has finite dimensional de Rham cohomology. From (4) of Lemma \ref{DirectImageLemma} we obtain isomorphisms
\[
H^{i}_{\mathrm{dR}}(X,\mathscr{M})=H_{\mathrm{dR}}^{i+c}(\BB^{n},i_{+}\mathscr{M})
\]
where $c$ is the codimension of $X$ in $\BB^{n}$. This finishes the proof.
\end{proof}

\section{Proof of the Main Theorem \ref{MainThm}}

We finish the paper with the proof of Theorem \ref{MainThm}. For the  convenience of the reader we also recall some special cases, which we have already proven in the previous sections.

\begin{proof}[Proof of Theorem \ref{MainThm}]
If $X=\BB^{n}$ and $\mathscr{M}$ is globally finitely presented then the assertion follows from Proposition \ref{TateDisc}. If $X$ admits a global coordinate system and $\mathscr{M}$ is still globally finitely presented then we can consider some closed embedding $i:X\hookrightarrow \BB^{n}$ and the claim for $\mathscr{M}$ follows from the previous case applied to $i_{+}\mathscr{M}$. This is precisely Proposition \ref{GlobalCoordinates}. 

We now let $X$ to be a smooth quasi-compact, quasi-separated rigid analytic variety and $\mathscr{M}$ a (not necessarily globally finitely presented) holonomic $\mathscr{D}_{X}$-module. We first prove Theorem \ref{MainThm}~ under the additional assumption that $X$ is separated. There exists an affinoid open cover $X=\bigcup_{i=1}^{N}U_{i}$ such that each $U_{i}$ admits a global coordinate system and $\mathscr{M}_{|U_{i}}$ is a globally finitely presented left $\mathscr{D}_{U_{i}}$-module (it may be taken to be finite because $X$ is quasi-compact). As $X$ is separated and $U_{i}$ are affinoid each finite intersection $U_{i_{1}}\cap\dots\cap U_{i_{k}}$ is also an open affinoid in $X$. Note that such finite intersection in fact admits a global coordinate system and $\mathscr{M}_{|U_{i_{1}}\cap\dots\cap U_{i_{k}}}$ is a globally finitely presented $\mathscr{D}_{U_{i_{1}}\cap\dots\cap U_{i_{k}}}$-module. We now consider the spectral sequence from Lemma \ref{deRhamSpectralSequence} associated to the cover $\{U_{i}\}$. We have
\[
E_{1}^{p,q}=\bigoplus_{1\leq i_{1},\dots,i_{p}\leq N}H^{q}_{\mathrm{dR}}(U_{i_{1}}\cap\dots\cap U_{i_{p}},\mathscr{M}_{|U_{i_{1}}\cap\dots\cap U_{i_{p}}})\implies H^{p+q}_{\mathrm{dR}}(X,\mathscr{M}).
\]
Now $U_{i}$ are chosen so that
\[
\dim_{K}H^{q}_{\mathrm{dR}}(U_{i_{1}}\cap\dots\cap U_{i_{p}},\mathscr{M}_{|U_{i_{1}}\cap\dots\cap U_{i_{p}}})<\infty
\]
by Proposition \ref{GlobalCoordinates}. Since $E^{p,q}_{\infty}$ must be a subquotient of $E^{p,q}_{1}$ we conclude that it is a finite dimensional $K$-vector space. Therefore $H^{p+q}_{\mathrm{dR}}(X,\mathscr{M})$ admits a finite filtration by finitely dimensional $K$-vector spaces and hence is of finite dimension.

Finally we prove Theorem \ref{MainThm} in full generality. The argument is essentially a repetition of the argument above. Under our assumptions on $X$ there exists a finite open cover $X=\bigcup_{i=1}^{N}U_{i}$ such that each $U_{i}$ is separated. Then every finite intersection $U_{i_{1}}\cap\dots\cap U_{i_{k}}$ is again separated. Therefore we can use the spectral sequence of Lemma \ref{deRhamSpectralSequence} associated to this open cover to deduce finiteness of the de Rham cohomology from the case when $X$ is separated.
\end{proof}

\bibliographystyle{plain}
\bibliography{references}

\end{document}